\documentclass[11pt,reqno]{amsart}

\usepackage{amsmath,amsthm,amsfonts,amssymb,latexsym,mathrsfs,color}
\usepackage{hyperref}

\newtheorem{theorem}{Theorem}

\newtheorem{definition}[theorem]{Definition}

\newcommand{\ap}{{\rm ap\,}}

\newcommand{\ipk}{{\rm ipk\,}}
\newcommand{\asc}{{\rm asc\,}}
\newcommand{\lpk}{{\rm lpk\,}}

\newcommand{\run}{{\rm run\,}}
\newcommand{\cyc}{{\rm cyc\,}}

\newcommand{\exc}{{\rm exc\,}}

\newcommand{\msn}{\mathfrak{S}_n}
\newcommand{\mqn}{\mathcal{Q}_n}
\newcommand{\msnn}{\mathfrak{S}_{2n}}
\newcommand{\z}{ \mathbb{Z}}

\newcommand{\as}{{\rm as\,}}

\newcommand{\stirling}[2]{\genfrac{[}{]}{0pt}{}{#1}{#2}}

\title{The $1/k$-Eulerian polynomials and $k$-Stirling permutations}
\author[S.-M.~Ma]{Shi-Mei~Ma}
\address{School of Mathematics and Statistics,
        Northeastern University at Qinhuangdao,
         Hebei 066004, P. R. China}
\email{shimeimapapers@163.com (S.-M. Ma)}
\author[T.~Mansour]{Toufik Mansour}
\address{Department of Mathematics, University of Haifa, 3498838 Haifa, Israel}
\email{toufik@math.haifa.ac.il (T. Mansour)}

\begin{document}
\maketitle

\begin{abstract}
In this paper, we establish a connection between the $1/k$-Eulerian polynomials introduced by Savage and Viswanathan ({\newblock Electron. J. Combin. 19 (2012), \#P9}) and $k$-Stirling permutations. We also introduce the dual set of Stirling permutations.
\bigskip

\end{abstract}
\date{\today}

\section{Introduction}
For $k\geq 1$, the {\it $1/k$-Eulerian polynomials} $A_n^{(k)}(x)$ are defined by
\begin{equation}\label{Ankx-def01}
\sum_{n\geq 0}A_n^{(k)}(x)\frac{z^n}{n!}=\left(\frac{1-x}{e^{kz(x-1)}-x} \right)^{\frac{1}{k}}.
\end{equation}

Let $e=(e_1,e_2,\ldots,e_n)\in\z^n$. Let
$I_{n,k}=\left\{ e|0\leq e_i\leq (i-1)k\right\}$, which known as the set of $n$-dimensional {\it $k$-inversion sequences}.
The number of {\it ascents} of $e$ is defined by
$$\asc(e)=\#\left\{i:1\leq i\leq n-1\mid\frac{e_i}{(i-1)k+1}<\frac{e_{i+1}}{ik+1}\right\}.$$
Savage and Viswanathan~\cite{Savage12} showed that
\begin{equation}\label{Ankx-def02}
A_n^{(k)}(x)=\sum_{e\in I_{n,k}}x^{\asc(e)}.
\end{equation}

Let $\msn$ be the
symmetric group on the set $[n]=\{1,2,\ldots,n\}$ and
$\pi=\pi_1\pi_2\cdots\pi_n\in\msn$.
The number of {\it
excedances} of $\pi$ is $\exc(\pi):=\#\{i:1\leq i\leq n-1|\pi_i>i\}$.
Let $\cyc(\pi)$ be the number of {\it cycles} in the disjoint cycle representation of $\pi$.
In~\cite{FS70},
Foata and Sch\"utzenberger introduced a $q$-analog of the classical
Eulerian polynomials defined by
\begin{equation}\label{anxq-def}
A_n(x;q)=\sum_{\pi\in\msn}x^{\exc(\pi)}q^{\cyc(\pi)}.
\end{equation}
The polynomials $A_n(x;q)$ satisfy the recurrence relation
\begin{equation}\label{anxq-rr}
A_{n+1}(x;q)=(nx+q)A_{n}(x;q)+x(1-x)\frac{d}{d x}A_{n}(x;q),
\end{equation}
with the initial conditions $A_{1}(x;q)=1$ and $A_{2}(x;q)=q$ (see~\cite[Proposition~7.2]{Brenti00}).
Savage and Viswanathan~\cite[Section~1.5]{Savage12} discovered that
\begin{equation}\label{Ankx-def03}
A_n^{(k)}(x)=k^nA_n(x;1/k)=\sum_{\pi\in\msn}x^{\exc(\pi)}k^{n-\cyc(\pi)}.
\end{equation}
Let $A_n^{(k)}(x)=\sum_{j=0}^{n-1}a_{n,j}^{(k)}x^j$. It follows from~\eqref{anxq-rr} and~\eqref{Ankx-def03} that
\begin{equation}\label{anj-rr}
a_{n+1,j}^{(k)}=(1+kj)a_{n,j}^{(k)}+k(n-j+1)a_{n,j-1}^{(k)},
\end{equation}
with the initial condition $a_{1,0}^{(k)}=1$.

Let $\stirling{n}{k}$ be the Stirling number of the first kind, i.e., the number of
permutations in $\msn$ with precisely $k$ cycles. It is well known that
\begin{equation*}\label{Stirlingnumbers}
\sum_{k=0}^n\stirling{n}{k}x^k=\prod_{i=0}^{n-1}(x+i).
\end{equation*}
Thus it follows from~\eqref{Ankx-def03} that
\begin{equation*}\label{Ank1}
A_n^{(k)}(1)=\prod_{i=1}^{n-1}(ik+1)\quad\textrm{for $n\ge 1$}.
\end{equation*}
Since $\prod_{i=1}^{n-1}(ik+1)$ also count $k$-Stirling permutations of order $n$ (see~\cite{Janson11,Kuba12}), it is natural
to consider the following question:  Is there existing a connection between $A_n^{(k)}(x)$ and $k$-Stirling permutations?
The main object of this paper is to provide a solution to this problem.

\section{$k$-Stirling permutations and their longest ascent-plateau}
In the following discussion, we always let $j^i=\underbrace{j,j,\ldots,j}_i$ for $i,j\geq 1$.
Stirling permutations were defined by Gessel
and Stanley~\cite{Gessel78}. A {\it Stirling permutation} of order $n$ is a permutation of the multiset $\{1^2,2^2,\ldots,n^2\}$ such that
for each $i$, $1\leq i\leq n$, all entries between the two occurrences of $i$ are larger than $i$.
We call a permutation of the multiset
$\{1^k,2^k,\ldots,n^k\}$ a $k$-{\it Stirling permutation} of order $n$ if for each $i$, $1\leq i\leq n$,
all entries between the two occurrences of $i$ are at least $i$.
Denote by $\mqn(k)$ the set of $k$-{\it Stirling permutation} of order $n$.
Clearly, $\mqn(1)=\msn$ and $\mqn(2)$ is the set of ordinary Stirling permutations of order $n$.

For $\sigma=\sigma_1\sigma_2\cdots\sigma_{2n}\in\mqn(2)$, an occurrence of an {\it ascent} (resp. {\it plateau}) is an index $i$ such that $\sigma_i<\sigma_{i+1}$ (resp. $\sigma_i=\sigma_{i+1}$).
The reader is referred to~\cite{Bona08,Janson11,Kuba12,Remmel14} for recent
progress on the study of patterns in Stirling permutations.

\begin{definition}
Let $\sigma=\sigma_1\sigma_2\cdots\sigma_{kn}\in \mqn(k)$.
We say that an index $i\in \{2,3,\ldots,nk-k+1\}$ is a longest ascent-plateau if $$\sigma_{i-1}<\sigma_i=\sigma_{i+1}=\sigma_{i+2}=\cdots=\sigma_{i+k-1}.$$
\end{definition}

Let $\ap(\sigma)$ be the number of the longest ascent-plateau of $\sigma$.
For example, $\ap(1122\textbf{3}3321)=1$.

Now we present the main results of this paper.
\begin{theorem}\label{thm:01}
For $n\geq 1$ and $k\geq 2$, we have
\begin{equation*}\label{Ankx-stirling}
A_n^{(k)}(x)=\sum_{\sigma\in \mqn(k)}x^{\ap(\sigma)}.
\end{equation*}
\end{theorem}
\begin{proof}
Let $$T(n,j;k)=\#\{\sigma\in \mqn(k):\ap(\sigma)=j\}.$$
There are two ways in which a permutation $\widetilde{\sigma}\in\mathcal{Q}_{n+1}(k)$ with the number of the longest ascent-plateau equals $j$
can be obtained from a permutation $\sigma\in\mathcal{Q}_{n}(k)$.
\begin{enumerate}
\item [(a)] If the number of the longest ascent-plateau of $\sigma$ equals $j$, then we can insert $k$ copies of $(n+1)$ into $\sigma$
without increasing the number of the longest ascent-plateau.
Let $i$ be one of the longest ascent-plateau of $\sigma$. Then we can
insert $k$ copies of $(n+1)$ before $\sigma_i$ or after $\sigma_{t}$, where $i\leq t\leq i+k-2$.
Moreover, the $k$ copies of $(n+1)$ can also be inserted into the front of $\sigma$.
This accounts for $(1+kj)T(n,j;k)$ possibilities.
\item [(b)] If the number of the longest ascent-plateau of $\sigma$ equals $j-1$,
then we insert $k$ copies of $(n+1)$ into the remaining $1+kn-(1+k(j-1))=k(n-j+1)$ positions.
This gives $k(n-j+1)T(n,j-1;k)$ possibilities.
\end{enumerate}
Hence
$$T(n+1,j;k)=(1+kj)T(n,j;k)+k(n-j+1)T(n,j-1;k).$$
Clearly, $T(n,0;k)=1$, corresponding to the permutation $n^k(n-1)^k\cdots1^k$.
Therefore, the numbers $T(n,j;k)$ satisfy the same recurrence relation
and initial conditions as $a_{n,j}^{(k)}$, so they agree.
\end{proof}

Define $${\mathcal{Q}}^0_{n}(k)=\{0\sigma: \sigma\in\mqn(k)\}.$$
Therefore, for $\sigma\in{\mathcal{Q}}^0_{n}(k)$, we let $\sigma_0=0$ and the indices of the longest ascent-plateau belong to $\{1,2,3,\ldots,nk-k+1\}$.
For example,
$\ap(0\textbf{1}12\textbf{3}32)=2$.

Define $$x^nA_n^{(k)}\left(\frac{1}{x}\right)=\sum_{j=1}^nb_{n,j}^{(k)}x^j.$$
Then $b_{n,j}^{(k)}=a_{n,n-j}^{(k)}$. It follows from~\eqref{anj-rr} that
$$b_{n+1,j}^{(k)}=kjb_{n,j}^{(k)}+(kn-kj+k+1)b_{n,j-1}^{(k)}.$$
Along the same lines of the proof of Theorem~\ref{thm:01}, we get the following result.
\begin{theorem}\label{thm:02}
For $n\geq 1$ and $k\geq 2$, we have
\begin{equation*}\label{Ankn-j}
x^nA_n^{(k)}\left(\frac{1}{x}\right)=\sum_{\sigma\in {\mathcal{Q}}^0_{n}(k)}x^{\ap(\sigma)}.
\end{equation*}
\end{theorem}

\section{The dual set of Stirling permutations}
For convenience, we let $\mqn=\mqn(2)$.
Let $\sigma=\sigma_1\sigma_2\cdots\sigma_{2n}\in\mqn$.
Let $\Phi$ be the bijection which map each first occurrence of letter $j$ in $\sigma$ to $2j$ and the
second occurrence of letter $j$ in $\sigma$ to $2j-1$,
where $j\in [n]$. For example, $\Phi(221331)=432651$.
The {\it dual set} $\Phi(\mqn)$ of $\mqn$ is defined by $$\Phi(\mqn)=\{\pi: \sigma\in\mqn, \Phi(\sigma)=\pi\}.$$
Clearly, $\Phi(\mqn)$ is a subset of $\msnn$. Let $ab$ be an ascent in $\sigma$, so $a<b$.
Using $\Phi$, we see that $ab$ is maps into $(2a-1)(2b-1)$, $(2a-1)(2b)$, $(2a)(2b-1)$ or $(2a)(2b)$, and vice versa.
Let $\as(\sigma)$ (resp. $\as(\pi)$) be the number of ascents of $\sigma$ (resp. $\pi$).
Then $\Phi$ preserving ascents, i.e., $\as(\sigma)=\as(\Phi(\sigma))=\as(\pi)$.
Hence the well known {\it Eulerian polynomial of second kind} $P_n(x)$ (see~\cite[A008517]{Sloane}) has the expression
$$P_n(x)=\sum_{\pi\in\Phi(\mqn)}x^{\as(\pi)}.$$

Perhaps one of the most important permutation statistics is the peaks statistic; see, e.g., \cite{Dilks09,Ma12} and the references contained therein.
Let $\pi=\pi_1\pi_2\cdots\pi_n\in\msn$. An {\it interior peak} in $\pi$ is
an index $i\in\{2,3,\ldots,n-1\}$ such that $\pi_{i-1}<\pi_i>\pi_{i+1}$.
Let $\ipk(\pi)$ denote the number of interior peaks in $\pi$.
A {\it left peak} in $\pi$ is an index $i\in[n-1]$ such that $\pi_{i-1}<\pi_i>\pi_{i+1}$, where we take $\pi_0=0$.
Denote by $\lpk(\pi)$ the number of left peaks in $\pi$. For example, $\ipk(21435)=1$ and $\lpk(21435)=2$.

As pointed out by Savage and Viswanathan~\cite[Section~4]{Savage12} that the numbers $a_{n,j}^{(2)}$ appear as A185410 in~\cite{Sloane},
and the numbers $a_{n,n-j}^{(2)}$ appear as A156919 in~\cite{Sloane}.
We can now present a unified characterization of these numbers.
\begin{theorem}
For $n\geq 1$, we have
\begin{equation}\label{An2x}
A_n^{(2)}(x)=\sum_{\pi\in\Phi(\mqn)}x^{\ipk(\pi)},
\end{equation}
\begin{equation}\label{An2x-reverse}
x^nA_n^{(2)}\left(\frac{1}{x}\right)=\sum_{\pi\in\Phi(\mqn)}x^{\lpk(\pi)}.
\end{equation}
\end{theorem}
\begin{proof}
Recall that an occurrence of a {\it pattern} $\tau$ in a sequence $\pi$ is defined as a subword
in $\pi$ whose letters are in the same relative order as those in $\tau$.

Let $\sigma\in\mqn$ and let $\Phi(\sigma)=\pi$.
Let $$C=\{112,211,122,221,213,312,123,321\}.$$
For all $\sigma\in\mqn$, we see that all patterns of length three of $\sigma$ are belong to $C$.
Let $abb$ be an occurrence of the pattern $122$ in $\sigma$, so $a<b$.
Using $\Phi$, we see that $abb$ is maps to either $(2a-1)(2b)(2b-1)$ or $(2a)(2b)(2b-1)$, which is an interior peak of the pattern $132$.
Moreover, one can easily verify that interior peaks can not be generated by the other patterns.
Recall that an occurrence of the longest ascent-plateau in Stirling permutations is an occurrence of the pattern $122$.
Then we get~\eqref{An2x} by using Theorem~\ref{thm:01}. Similarly, from Theorem~\ref{thm:02}, we get~\eqref{An2x-reverse}.
\end{proof}

For $n\geq 1$, we define $C_n(x)$ by
\begin{equation}\label{def:Cnx}
(1+x)C_n(x)=xA_n^{(2)}(x^2)+x^{2n}A_n^{(2)}\left(\frac{1}{x^2}\right).
\end{equation}
Set $C_0(x)=1$.
It follows from~\eqref{Ankx-def01} that
$$C(x,z)=\sum_{n\geq0} C_n(x)\frac{z^n}{n!}=\frac{{e^{z \left( x-1 \right)  \left( 1+x \right) }}+x}{1+x}\sqrt {{\frac {1-x^2}{{e^{2\,z \left( x-1
\right) \left( 1+x \right) }}-x^2}}}.$$

The first few $C_n(x)$ are given as follows:
\begin{align*}
  C_1(x)&=x, \\
  C_2(x)&=x+x^2+x^3, \\
  C_3(x)&=x+3x^2+7x^3+3x^4+x^5,\\
  C_4(x)&=x+7x^2+29x^3+31x^4+29x^5+7x^6+x^7,\\
  C_5(x)&=x+15x^2+101x^3+195x^4+321x^5+195x^6+101x^7+15x^8+x^9.
\end{align*}

Let $\pi=\pi_1\pi_2\cdots\pi_n\in\msn$.
We say that $\pi$ changes direction at position $i$ if either $\pi_{i-1}<\pi_i>\pi_{i+1}$, or
$\pi_{i-1}>\pi_i<\pi_{i+1}$, where $i\in\{2,3,\ldots,n-1\}$. We say that $\pi$ has $k$ {\it alternating runs} if there are $k-1$ indices $i$ where $\pi$ changes direction~(see~\cite[A059427]{Sloane}). Let $\run(\pi)$ denote the number of alternating runs of $\pi$. For example, $\run(214653)=3$.
There is a large literature devoted to the distribution of alternating runs. The reader is referred to~\cite{CW08,Ma122} for recent results on this subject.

We can now conclude the following result.
\begin{theorem}
For $n\geq 1$, we have
\begin{equation}\label{Cnx:run}
C_n(x)=\sum_{\pi\in\Phi(\mqn)}x^{\run(\pi)}.
\end{equation}
\end{theorem}
\begin{proof}
Define
\begin{align*}
  S_1& =\{\pi\in \Phi(\mqn): \lpk(\pi)=\ipk(\pi)\}, \\
  S_2& =\{\pi\in \Phi(\mqn): \lpk(\pi)=\ipk(\pi)+1\}.
\end{align*}
Then $\Phi(\mqn)$ can be partitioned into subsets $S_1$ and $S_2$.

From~\eqref{def:Cnx}, we have
\begin{align*}
(1+x)C_n(x)& =\sum_{\pi\in\Phi(\mqn)}x^{2\ipk(\pi)+1}+\sum_{\pi\in\Phi(\mqn)}x^{2\lpk(\pi)}\\
           & =x\sum_{\pi\in S_1}x^{\ipk(\pi)+\lpk(\pi)}+\sum_{\pi\in S_2}x^{\ipk(\pi)+\lpk(\pi)}+\sum_{\pi\in S_1}x^{\ipk(\pi)+\lpk(\pi)}+\\
           &x\sum_{\pi\in S_2}x^{\ipk(\pi)+\lpk(\pi)}\\
           &=(1+x)\sum_{\pi\in S_1}x^{\ipk(\pi)+\lpk(\pi)}+(1+x)\sum_{\pi\in S_2}x^{\ipk(\pi)+\lpk(\pi)}.
\end{align*}
Thus
$$C_n(x)=\sum_{\pi\in\Phi(\mqn)}x^{\ipk(\pi)+\lpk(\pi)}.$$
Note that all $\pi\in\Phi(\mqn)$ ends with a descent, i.e., $\pi_{2n-1}>\pi_{2n}$.
Hence~\eqref{Cnx:run} follows from the fact that $\run(\pi)=\ipk(\pi)+\lpk(\pi)$.
\end{proof}

\section{Concluding remarks}
It follows from~\eqref{Ankx-def02} and Theorem~\ref{thm:01}, we have
\begin{equation}\label{problem:01}
\sum_{e\in I_{n,k}}x^{\asc(e)}=\sum_{\sigma\in \mqn(k)}x^{\ap(\sigma)}.
\end{equation}
Combining~\eqref{Ankx-def03} and Theorem~\ref{thm:01}, we have
\begin{equation}\label{problem:02}
\sum_{\pi\in\msn}x^{\exc(\pi)}k^{n-\cyc(\pi)}=\sum_{\sigma\in \mqn(k)}x^{\ap(\sigma)}.
\end{equation}
It would be interesting to present a combinatorial proof of~\eqref{problem:01} or~\eqref{problem:02}.


\end{document}